\documentclass[12pt]{amsart}
\usepackage{amsmath,amssymb,amsbsy,amsfonts,latexsym,amsopn,amstext,cite,
                                               amsxtra,euscript,amscd,bm,mathabx}
\usepackage{url}
\usepackage[colorlinks,linkcolor=blue,anchorcolor=blue,citecolor=blue,backref=page]{hyperref}
\usepackage{color}
\usepackage{graphics,epsfig}
\usepackage{graphicx}
\usepackage{float} 
\usepackage[english]{babel} 
\usepackage{mathtools}
\usepackage{todonotes}
\usepackage{url}
\usepackage[colorlinks,linkcolor=blue,anchorcolor=blue,citecolor=blue,backref=page]{hyperref}

\def\md{{\sf{m.d.}}}

\usepackage[norefs,nocites]{refcheck}

\hypersetup{breaklinks=true}

\usepackage[norefs,nocites]{refcheck}
\usepackage[english]{babel}
\begin{document}

\newtheorem{thm}{Theorem}
\newtheorem{lem}[thm]{Lemma}
\newtheorem{claim}[thm]{Claim}
\newtheorem{cor}[thm]{Corollary}
\newtheorem{prop}[thm]{Proposition} 
\newtheorem{definition}[thm]{Definition}
\newtheorem{rem}[thm]{Remark} 
\newtheorem{question}[thm]{Open Question}
\newtheorem{qn}[thm]{Question}
\newtheorem{conj}[thm]{Conjecture}
\newtheorem{prob}{Problem}

\newcommand{\GL}{\operatorname{GL}}
\newcommand{\SL}{\operatorname{SL}}
\newcommand{\lcm}{\operatorname{lcm}}
\newcommand{\ord}{\operatorname{ord}}
\newcommand{\Op}{\operatorname{Op}}
\newcommand{\Tr}{\operatorname{Tr}}
\newcommand{\Nm}{\operatorname{Nm}}

\numberwithin{equation}{section}
\numberwithin{thm}{section}
\numberwithin{table}{section}

\def\vol {{\mathrm{vol\,}}}
\def\squareforqed{\hbox{\rlap{$\sqcap$}$\sqcup$}}
\def\qed{\ifmmode\squareforqed\else{\unskip\nobreak\hfil
\penalty50\hskip1em\null\nobreak\hfil\squareforqed
\parfillskip=0pt\finalhyphendemerits=0\endgraf}\fi}

\def \balpha{\bm{\alpha}}
\def \bbeta{\bm{\beta}}
\def \bgamma{\bm{\gamma}}
\def \blambda{\bm{\lambda}}
\def \bchi{\bm{\chi}}
\def \bphi{\bm{\varphi}}
\def \bpsi{\bm{\psi}}
\def \bomega{\bm{\omega}}
\def \btheta{\bm{\vartheta}}
\def \ochi{\overline{\chi}}
\def \h{\widehat h}
\def \whD{\widehat D}
\def \whX{\widehat X}

\def\ovG {\overline{\Gamma}}

\def\uu {\mathbf{u}} 
\def\uv {\mathbf{v}} 
\def\uk {\mathbf{k}} 
\def\ut {\mathbf{t}}

\def\eps{\varepsilon}

\newcommand{\bfxi}{{\boldsymbol{\xi}}}
\newcommand{\bfrho}{{\boldsymbol{\rho}}}

\def\Kab{\sfK_\psi(a,b)}
\def\Kuv{\sfK_\psi(u,v)}
\def\SaUV{\cS_\psi(\balpha;\cU,\cV)}
\def\SaAV{\cS_\psi(\balpha;\cA,\cV)}

\def\SUV{\cS_\psi(\cU,\cV)}
\def\SAB{\cS_\psi(\cA,\cB)}

\def\Kmnp{\sfK_p(m,n)}

\def\KKap{\cH_p(a)}
\def\KKaq{\cH_q(a)}
\def\KKmnp{\cH_p(m,n)}
\def\KKmnq{\cH_q(m,n)}

\def\Klmnp{\sfK_p(\ell, m,n)}
\def\Klmnq{\sfK_q(\ell, m,n)}

\def \SALMNq {\cS_q(\balpha;\cL,\cI,\cJ)}
\def \SALMNp {\cS_p(\balpha;\cL,\cI,\cJ)}

\def \SACXMQX {\fS(\balpha,\bzeta, \bxi; M,Q,X)}

\def\SAMJp{\cS_p(\balpha;\cM,\cJ)}
\def\SAMJq{\cS_q(\balpha;\cM,\cJ)}
\def\SAqMJq{\cS_q(\balpha_q;\cM,\cJ)}
\def\SAJq{\cS_q(\balpha;\cJ)}
\def\SAqJq{\cS_q(\balpha_q;\cJ)}
\def\SAIJp{\cS_p(\balpha;\cI,\cJ)}
\def\SAIJq{\cS_q(\balpha;\cI,\cJ)}

\def\RIJp{\cR_p(\cI,\cJ)}
\def\RIJq{\cR_q(\cI,\cJ)}

\def\TWXJp{\cT_p(\bomega;\cX,\cJ)}
\def\TWXJq{\cT_q(\bomega;\cX,\cJ)}
\def\TWpXJp{\cT_p(\bomega_p;\cX,\cJ)}
\def\TWqXJq{\cT_q(\bomega_q;\cX,\cJ)}
\def\TWJq{\cT_q(\bomega;\cJ)}
\def\TWqJq{\cT_q(\bomega_q;\cJ)}

 \def \xbar{\overline x}
  \def \ybar{\overline y}

\def\cA{{\mathcal A}}
\def\cB{{\mathcal B}}
\def\cC{{\mathcal C}}
\def\cD{{\mathcal D}}
\def\cE{{\mathcal E}}
\def\cF{{\mathcal F}}
\def\cG{{\mathcal G}}
\def\cH{{\mathcal H}}
\def\cI{{\mathcal I}}
\def\cJ{{\mathcal J}}
\def\cK{{\mathcal K}}
\def\cL{{\mathcal L}}
\def\cM{{\mathcal M}}
\def\cN{{\mathcal N}}
\def\cO{{\mathcal O}}
\def\cP{{\mathcal P}}
\def\cQ{{\mathcal Q}}
\def\cR{{\mathcal R}}
\def\cS{{\mathcal S}}
\def\cT{{\mathcal T}}
\def\cU{{\mathcal U}}
\def\cV{{\mathcal V}}
\def\cW{{\mathcal W}}
\def\cX{{\mathcal X}}
\def\cY{{\mathcal Y}}
\def\cZ{{\mathcal Z}}
\def\Ker{{\mathrm{Ker}}}
\def\g{{\mathrm{gcd}}}

\def\Psf {\mathsf P}
\def\Qsf {\mathsf Q}
\def\fsf {\mathsf f}
\def\Dsf {\mathsf D}

\def\NmQR{N(m;Q,R)}
\def\VmQR{\cV(m;Q,R)}

\def\Xm{\cX_m}

\def \A {{\mathbb A}}
\def \B {{\mathbb A}}
\def \C {{\mathbb C}}
\def \N {{\mathbb N}}
\def \F {{\mathbb F}}
\def \G {{\mathbb G}}
\def \L {{\mathbb L}}
\def \K {{\mathbb K}}
\def \PP {{\mathbb P}}
\def \Q {{\mathbb Q}}
\def \R {{\mathbb R}}
\def \Z {{\mathbb Z}}
\def \fS{\mathfrak S}

\def\e{{\mathbf{\,e}}}
\def\ep{{\mathbf{\,e}}_p}
\def\eq{{\mathbf{\,e}}_q}
\def\er{{\mathbf{\,e}}_R}
\def\\{\cr}
\def\({\left(}
\def\){\right)}
\def\fl#1{\left\lfloor#1\right\rfloor}
\def\rf#1{\left\lceil#1\right\rceil}

\def\Tr{{\mathrm{Tr}}}
\def\Nm{{\mathrm{Nm}}}
\def\Im{{\mathrm{Im}}}

\def \oF {\overline \F}

\newcommand{\pfrac}[2]{{\left(\frac{#1}{#2}\right)}}

\def \Prob{{\mathrm {}}}
\def\e{\mathbf{e}}
\def\ep{{\mathbf{\,e}}_p}
\def\epp{{\mathbf{\,e}}_{p^2}}
\def\em{{\mathbf{\,e}}_m}

\def\Res{\mathrm{Res}}
\def\Orb{\mathrm{Orb}}

\def\vec#1{\mathbf{#1}}
\def \va{\vec{a}}
\def \vb{\vec{b}}
\def \vm{\vec{m}}
\def \vu{\vec{u}}
\def \vv{\vec{v}}
\def \vx{\vec{x}}
\def \vy{\vec{y}}
\def \vz{\vec{z}}
\def\flp#1{{\left\langle#1\right\rangle}_p}
\def\T {\mathsf {T}}

\def\sfG {\mathsf {G}}
\def\sfK {\mathsf {K}}

\def\mand{\qquad\mbox{and}\qquad}

\title[Multiplicative dependence on  elliptic curves]
{Multiplicative dependence in the denominators of points of elliptic curves}

\author[A. B\' erczes] {Attila B\'erczes}
\address{Institute of Mathematics, University of Debrecen, P. O. Box 400, H-4002 Debrecen, Hungary}
\email{berczesa@science.unideb.hu}

\author[S. Bhakta]{Subham Bhakta}
\address{School of Mathematics and Statistics, University of New South Wales, Sydney, NSW 2052, Australia.} 
\email{subham.bhakta@unsw.edu.au}

\author[L.  Hajdu] {Lajos Hajdu}
\address{Count Istv{\'a}n Tisza Foundation, Institute of Mathematics, University of Debrecen, 
  P. O. Box 400, H-4002 Debrecen, Hungary, and 
HUN-REN-DE Equations, Functions, Curves
and their Applications Rese\-arch Group}
\email{hajdul@science.unideb.hu}

\author[A. Ostafe]{Alina Ostafe}
\address{School of Mathematics and Statistics, University of New South Wales, Sydney, NSW 2052, Australia.} 
\email{alina.ostafe@unsw.edu.au}

\author[I. E. Shparlinski]{Igor E. Shparlinski}
\address{School of Mathematics and Statistics, University of New South Wales, Sydney, NSW 2052, Australia.} 
\email{igor.shparlinski@unsw.edu.au}

 \dedicatory{Dedicated to the 75th birthday of Attila Peth{\H o}} 

\begin{abstract}
Let $E_1, \ldots, E_s $  be $s$, not necessary distinct, elliptic curves over $\Q$. We give upper bounds on the frequency of $s$-tuples of points in $E_1(\Q)\times \ldots \times E_s(\Q)$ whose denominators
 or $x$-coordinates are multiplicatively dependent. More precisely, we give such bounds in two scenarios: one in which we fix $s$ non-torsion  $\Q$-rational points $P_i \in E_i(\Q)$ and 
arbitrary  $\Q$-rational points $Q_i \in E_i(\Q)$,  $i =1, \ldots, s$, and we count $s$-tuples 
\[
(n_1P_1+Q_1,\ldots, n_sP_s+Q_s) \in  E_1(\Q) \times \ldots \times E_s(\Q)
\]
with $n_1, \ldots, n_s$ in an arbitrary interval of length $N$, and the second in which we count points $(P_1,\ldots,P_s) \in E_1(\Q) \times \ldots \times E_s(\Q)$ of bounded canonical height.
%
\end{abstract} 

\maketitle
\tableofcontents

\section{Introduction}
\subsection{Set-up}
For an  elliptic curve $E$ given by a  short Weierstrass equation 
\begin{equation}\label{eq:Weier}
y^2   = x^3 + a x + b
\end{equation}
with integral coefficients  
 $a$ and $b$, we denote by $E(\Q)$ the group of $\Q$-rational  points of $E$, and $O$ denotes the point at infinity, see~\cite{Silv-Book} for background.

We can write any point $P \in E(\Q)$, in the lowest form 
\[
P =\left(\frac{a_P}{d_P^2},\frac{b_P}{d_P^3}\right),
\]
where $d_P \in\mathbb{N}$, $a_P,b_P\in\Z$, and $\gcd(a_Pb_P , d_P)=1$.

As usual, we say that the nonzero complex numbers $\gamma_1,\ldots,\gamma_s$ are {\it multiplicatively dependent} (\md\/), if there exist integers $k_1,\ldots,k_s$, not all zero, such that
\[
\gamma_1^{k_1}\ldots \gamma_s^{k_s}=1.
\]
We also say that $\gamma_1,\ldots,\gamma_s$ are \md\ {\it of maximal rank} if no sub-tuple of $(\gamma_1,\ldots,\gamma_s)$ is \md.

Now assume we are given  $s$, not necessary distinct, elliptic curves  $E_1, \ldots, E_s $  over $\Q$ of positive rank. 
Given $s$-tuples 
\[
\Psf=(P_1, \ldots, P_s) \mand \Qsf=(Q_1, \ldots, Q_s)  
\]
of non-torsion points $P_i\in  E_i(\Q)$  and arbitrary 
points $Q_i \in E_i(\Q)$, $i =1, \ldots, s$, 
 we are interested in estimating the following quantities
\begin{align*}
D_{\Psf,\Qsf}(M,N) =\sharp\, \{ (n_1, \ldots, n_s)&\in   (M,M+N]^s:\\
& d_{n_1P_1 + Q_1}, \ldots, d_{n_sP_s+Q_s}\ \text{are \md} \}, 
\end{align*} 
and 
\begin{align*}
X_{\Psf,\Qsf}(M,N) =\sharp\, \{ (n_1&, \ldots, n_s)\in   (M,M+N]^s:\\
&  x(n_1P_1 + Q_1), \ldots, x(n_sP_s+Q_s)\  \text{are \md} \}.
\end{align*}

To estimate  $D_{\Psf,\Qsf}(M,N)$ and $X_{\Psf,\Qsf}(M,N)$, it is enough to estimate
\begin{align*}
D^*_{\Psf,\Qsf}(M,N) =\sharp\, \{ (n_1&, \ldots, n_s)\in   (M,M+N]^s:\\
&  d_{n_1P_1 + Q_1}, \ldots, d_{n_sP_s+Q_s}\ \text{are \md\ of maximal rank} \},
\end{align*} 
and 
\begin{align*}
X^*_{\Psf,\Qsf}(M,N) =\sharp\, \{ (n_1&, \ldots, n_s)\in   (M,M+N]^s:\\
&  x(n_1P_1 + Q_1), \ldots, x(n_sP_s+Q_s)\\
& \qquad \qquad  \text{are \md\ of maximal rank} \}.
\end{align*}

In particular, we can assume that $k_1\ldots k_s \ne 0$, and also that the integers $n_i$ are pairwise distinct. We can then estimate $D_{\Psf,\Qsf}(M,N)$ via the inequality
\begin{equation}
\label{eq:D D*}
D_{\Psf,\Qsf}(M,N)\le  \sum_{j=1}^s \binom{s}{j} D^*_{\Psf,\Qsf}(M,N)  N^{s-j} , 
\end{equation}
and similarly for $X_{\Psf,\Qsf}(M,N)$.

We remark that when $E_1=\ldots = E_s$ and  $Q_i$, $i =1, \ldots, s$, 
 are all torsion points (including the points at infinity $Q_i = O_i \in E_i$)  then there are very strong 
versions of the Zsigmondy theorem on primitive prime divisors, that is, prime divisors that do not divide any previous term of the sequence, see, for example,~\cite{Silv88,Ver20,Ver21,Ver23}. 
In this case it is reasonably straightforward to analyse the behaviour of $D^*_{\Psf,\Qsf,s}(M,N)$.
Hence here we  concentrate on the general case.

Furthermore, we denote by $\h(P)$ the {\it canonical height\/} of a point $P$ on an elliptic curve $E$,  
see~\cite[Section~VIII.9]{Silv-Book}. If $E$ is of rank $r$ then by~\cite[Theorem 4.5]{Kowalski}, see also~\cite[Exercise~9.8.e]{Silv-Book},     for $H \ge 1$ we have 
\begin{equation}
\label{eq:Count height}
H^{r/2} \ll  \sharp\, \{P \in E(\Q):~\h(P) \le H\} \ll H^{r/2}, 
\end{equation} 
(which  follows from~\eqref{eq:h and ||} as well).

Now, assume we are given  $s$, not necessary distinct, elliptic curves  $E_1, \ldots, E_s $  over $\Q$
and define the following quantities
\begin{align*}
\whD^*(H) =\sharp\, \{ (P_1, \ldots, P_s):~P_i& \in E_{i}(\Q),\ \h(P_i) \le H, \ i =1, \ldots, s,\ \text{and}\\
&  d_{P_1},\ldots, d_{P_s}~\text{are \md\ of maximal rank}\},
\end{align*}
and
\begin{align*}
\whX^*(H) =\sharp\, \{ (P_1, \ldots, P_s):~P_i &\in E_{i}(\Q),\ \h(P_i) \le H, \ i =1, \ldots, s,\ \text{and}\\
  x(P_1) &,\ldots, x(P_s)~\text{are \md\ of maximal rank}\}.
\end{align*}

\section{Main results}

Since there are only finitely many integral points in $E(\Q)$, see~\cite[Chapter~IX, Corollary 3.2.2]{Silv-Book}, 
it is enough to estimate $D_{\Psf,\Qsf}^{*}(M,N)$  and $X_{\Psf,\Qsf}^{*}(M,N)$ for $s\geq 2$. 

We recall the following convention:  the notations  $U \ll V$ and $U = O(V)$, are equivalent
to $|U|  \le c V$ for some constant $c>0$,
which throughout the paper may depend on the points  $\Psf$ and $\Qsf$, and thus on the curves $E_1,\ldots,E_s$.
We now show that the multiplicative dependence of denominators and of $x$-coordinates of points 
\[
(n_1P_1+Q_1,\ldots, n_sP_s+Q_s) \in  E_1(\Q) \times \ldots \times E_s(\Q)
\]
as in the above  is quite rare.

\begin{thm}\label{thm:MultDep-EC-MN-Q}
    Let  $s\geq 2$ be a fixed integer. Then, uniformly over $M\geq 0$, we have 
\begin{equation}\label{eqn:D^*}
D^*_{\Psf,\Qsf}(M,N) \ll  N^{6s/7} \mand X^*_{\Psf,\Qsf}(M,N) \ll  N^{6s/7}.
\end{equation}
\end{thm}

Recall that by the Siegel theorem  there are only $O(1)$ values of $n$ with $ d_{nP + Q} = \pm 1$, corresponding to 
integer points on elliptic curves, see~\cite[Chapter~{IX}]{Silv-Book}.
We now see that the bottleneck in~\eqref{eq:D D*} comes from the case $s=2$. Hence, using 
Theorem~\ref{thm:MultDep-EC-MN-Q}, we have
the following bound on $D_{\Psf,\Qsf}(M,N)$ and $X_{\Psf,\Qsf}(M,N)$. 

\begin{cor}\label{cor:MultDep-EC-MN-Q}
    Let  $s\geq 2$ be a fixed integer. Then, uniformly over $M\geq 0$, we have 
\[
D_{\Psf,\Qsf}(M,N) \ll  N^{ s -2/7} \mand 
X_{\Psf,\Qsf}(M,N) \ll  N^{s -2/7}.
\]
\end{cor}

\begin{rem}\label{rem:s=2} In order to improve the bound of Corollary~\ref{cor:MultDep-EC-MN-Q}
one needs to get a better bound for the case $s=2$. Similarly to the argument in~\cite{BHOS}, 
this leads to a question of estimating the frequency of perfect powers in the sequences $d_{nP+Q}$.
If $Q = O$, then some finiteness results are provided by~\cite{HLS,NoSi}. However, as our 
Theorem~\ref{thm:MultDep-EC-MN-O} below shows, in this case we have a better bound anyway. 
\end{rem}

Next, in the special case when $Q_1, \ldots,  Q_s$ are all points at infinity we obtain a stronger 
bound when $s\le 6$. Namely,  let 
$\Psf=(P_1, \ldots, P_s)$ be an $s$-tuple of fixed non-torsion points with $P_i\in E_i(\Q)$, 
$i=1, \ldots, s$. We define
\begin{align*}
\Dsf_{\Psf}(M,N)=\sharp\, \{ (n_1,\ldots, n_s)\in &\left(M,M+N\right]^s :  \\ 
&  d_{n_1P_1},\ldots,d_{n_sP_s} \  \text{are \md} \}
\end{align*}
and prove the following result.  

\begin{thm}\label{thm:MultDep-EC-MN-O}
Let $s \geq 2$ be a fixed integer. Then, uniformly over $M \geq 0$, we have
\[
\Dsf_{\Psf}(M,N) \ll N^{s-1}.
\]
\end{thm}

Finally, we also have the following estimate in the case when the points are counted by canonical 
height.

\begin{thm}\label{thm:mdoveralldenos} 
    Let  $s\geq 2$ be a fixed integer and let  $E_1, \ldots, E_s $ be defined over $\Q$ whose ranks over  $\Q$ are $r_1,\ldots, r_s$,  respectively .  Then,  we have 
\[
\whD^*(H), \ \whX^*(H) \ll H^{(r_1+\ldots+r_s)/2-s/14}.
\]
\end{thm}

We note that most of our results can also be extended to counting 
\md\ numerators, we develop the necessary tools in Appendix~\ref{app:numer}.

\section{Preliminaries}
\label{sec:prelim}

\subsection{Size of denominators and numerators}
Now, we need  standard information about  the size of $a_{nP + Q}$ and $d_{nP+Q}$, regardless of whether $Q$ is a torsion point or not.    As before, let $\h$ be the canonical height function, see~\cite{Silv-Book}. 

By~\cite[Lemma~2.1]{EveShp}, we have the following asymptotic formula:

\begin{lem}\label{lem:height} 
For any    fixed points $P,Q \in E(\Q)$, we have 
\[
  \log (d_{nP + Q})=\( 0.5 \h(P)+ o(1)\) n^2, 
\]
as $n \to \infty$. 
\end{lem} 
 
Next, let  $P_{1}, \dots, P_{r} \in E(\mathbb{Q})$ be a fixed basis for the free part of $E(\Q)$.
For 
\[
P=\sum_{i=1}^{r} n_i P_i+T,
\]
where $T\in E(\Q)_{\mathrm{tor}} $ is a torsion point, we define
\[
\|P\|_{\infty}=\max\{|n_i|:~i =1, \ldots, s\}.
\]
Then, for any $P\in E(\Q)$, we have
\begin{equation}
\label{eq:h and ||} 
\|P\|^2_{\infty}\ll \widehat{h}(P)\ll \|P\|^2_{\infty},
\end{equation} 
see the proof of~\cite[Theorem~4.5]{Kowalski}, which in particular 
implies~\eqref{eq:Count height}.

\subsection{Congruences with denominators and numerators}
\label{sec:cong denom}

We first recall the following bound given by~\cite[Lemma~2.2]{EveShp}. 

\begin{lem}\label{lem:divbym-den} 
For an integer $m\ge 2$,  uniformly over $M\geq 0$, we have
\[
\sharp\,\{M<n\leq M+N :~m\mid d_{nP+Q}\}\ll \frac{N}{\sqrt{\log m}}+ 1.
\]
\end{lem}

Next, for a prime $p$, we denote by $\rho_p$ the index of appearance of $p$ as a divisor in the sequence $d_{nP}$, $n=1, 2, \ldots$, 
that is, the smallest $r$ such that $d_{rP} \equiv 0 \pmod p$; we set $\rho_p = \infty$ if no such $r$ exists, with the natural rules of operating 
with this quantity (like $\infty^{-1} = 0$). 

The following  result, only with the condition $p\mid d_{nP+Q}$,   has been  established in the proof of~\cite[Lemma~2.2]{EveShp}, and  if $Q = O$ is the point at infinity on $E$ it is also given as~\cite[Lemma~2.2]{Gott}.

Let $\nu_p$ denote  the $p$-adic valuation 
of rational numbers. 

\begin{lem}\label{lem:divbyp} 
    Let $p$ be any prime, then uniformly over points $Q \in E(\Q)$ we have
\[\sharp\,\{M<n\leq M+N :~\nu_p\(x(nP+Q)\) \ne 0 \}\ll \frac{N}{\rho_p}+ 1.
\]
\end{lem}  

\begin{proof} We count the number of $M<n\leq M+N$  such that $p \mid a_{nP+Q}$ and  $p \mid d_{nP+Q}$ separately, 
starting with $d_{nP+Q}$. 

Let $M +1 \le n_1 < \ldots < n_t \le M+N$ be all solutions to 
$ d_{nP+Q} \equiv 0 \pmod p$, $M<n\leq M+N$. If $t =1$ then there is nothing to prove. 

Otherwise there is $i =1, \ldots, t-1$ 
such that $n_{i+1} - n_i \le N/(t-1)$. 
We can also assume that $p$ is large enough so that the reduction of $E$ modulo $p$ 
is an elliptic curve  over the finite field of $p$ elements. 

Since 
\[
d_{n_iP+Q} \equiv d_{n_{i+1}P+Q} \equiv 0 \pmod p
\]
we see that $n_iP+Q$ and  $n_{i+1}P+Q$  are points at infinity in the reduction of $E$ modulo $p$, 
thus so is 
\[
\(n_{i+1}P+Q\) - \(n_iP+Q\) = \(n_{i+1}-n_i\)P.
\]
Thus $d_ {\(n_{i+1}-n_i\)P} \equiv 0 \pmod p$, which implies that $N/(t-1) \ge \rho_p$.  

Next let  $M +1 \le n_1 < \ldots < n_t \le M+N$ be all solutions to 
$ a_{n_iP+Q} \equiv 0 \pmod p$, $M<n\leq M+N$. Since for each $i=1,\ldots,t$, 
one has $y(n_iP+Q)^2 \equiv b \pmod p$, there are at most two values of 
$y(n_iP+Q) \bmod p$.  Therefore, there is a subsequence $m_1< \ldots < m_u$ 
 of the sequence $n_1 < \ldots < n_t$ of  length   $u  \ge t/2$ such that all 
 $m_1 P + Q \equiv  \ldots  \equiv  m_uP+Q\pmod p$. This means that $d_{(m_i-m_1)P}  \equiv 0 \pmod p$
 and recalling the above bound, we  conclude the proof. 
\end{proof} 

We also need the following version of~\cite[Lemma~2.1]{Gott}.

\begin{lem}\label{lem: rp a.a. p} 
 For any $R \ge 2$
\[\sharp\,  \{p :~\rho_p\le R \}\ll \frac{R^3}{\log R} .
\]
\end{lem}

\begin{rem} In~\cite{Gott}, the inequality of  Lemma~\ref{lem: rp a.a. p} is given with $R^3$. 
This is because the proof appeals to the bound $\omega(s) \ll  \log s$   on the number of prime
divisors of an integer $s \ge 2$. However, the trivial inequality $\omega(s)! \le   s$ and the  Stirling formula
immediately imply  $\omega(s) \ll  \log s / \log \log s $, which gives the present 
form of  Lemma~\ref{lem: rp a.a. p}. 
\end{rem} 

Let $E(\mathbb{Q})$ be of rank $r$ and let $p$ be a prime. We fix a basis 
 $\vec{P}=(P_1,\ldots, P_r)$ for its free part and 
 define
\begin{equation}\label{eqn:rE}
\rho_ p(\vec{P})=\min \{\rho_{i,p}:~  i=1, \ldots, r\},
    \end{equation}
where $\rho_{i,p}$ is defined as in above with respect to the points $P_i$, 
$i=1, \ldots, r$.

Denote
\[
\cE_{p}(H)=\{Q\in E(\Q) :~  \h(Q)\le H, \ \nu_p\(x(Q)\) \ne 0 \}.
\]

Then, we have the following version of Lemma~\ref{lem:divbyp}.

\begin{lem}\label{lem:divbypoverallpoints}
Let $p$ be any prime, then, in the above notation,  we have
\[
\sharp\, \cE_{p}(H) \ll H^{(r-1)/2}\(\frac{H^{1/2}}{\rho_ p(\vec{P})}+ 1\).
\]
\end{lem}

\begin{proof}
For each $ i=1, \ldots, r$, let us denote 
\begin{align*}
&\widetilde \cE_{i,p}(H)= \bigcup_{\substack{P\in E(\Q)/\langle P_{i} \rangle\\ \| P\|_{\infty}^2\le  H}} \left\{Q=nP_{i}+P:~|n| \le H^{1/2},~\nu_p(x(Q))\ne 0\right\},
\end{align*}
where  
\begin{equation}
\begin{split}
\label{eq:E/Pi}
& E(\Q)/\langle P_{i} \rangle\\
& \quad \,=
\left\{\sum_{\substack{  j=1\\j\ne i}}^r n_j P_{j}+T:~n_j \in \Z, \  1\le j \le r, \ j \ne i, \T\in E(\Q)_{\mathrm{tor}}\right\}.
\end{split}
\end{equation}

Now, we see from~\eqref{eq:h and ||}  that there is a constant $c>0$, depending only on the 
choice of $\vec{P}$, such that 
\[
\sharp\,\cE_{p}(H) \le \sum_{i=1}^{r} \sharp\,\widetilde \cE_{i,p}(cH).
\]
Applying  Lemma~\ref{lem:divbyp}  to each of the quantities $\sharp\,\widetilde \cE_{i,p}(cH)$, we conclude 
the proof. 
\end{proof}

\subsection{Counting $\cS$-units amongst the denominators}
Given a set $\cS$ of primes, we define by $\langle \cS\rangle \subseteq \Z$ the multiplicative semigroup 
generated by primes from  $\cS$. Consider the set
\[
\cU_{P,Q}(M, N;\cS)=\{M<n \leq M+N:~  d_{nP+Q} \in \langle \cS\rangle \}. 
\]

We argue as in the proof of~\cite[Theorem~1]{Shp}, and prove the following estimate.

\begin{lem}\label{lem:sqx}
For any finite  set $\cS$ of primes of cardinality $S = \sharp\,  \cS$, we have 
\[
\sharp\,  \cU_{P,Q}(M,N,\cS)\ll  \left(1+\frac{N}{M}\right)^2 S.
\]
 \end{lem} 

\begin{proof} 
Let us consider the product
\[
 W_d=\prod_{n\in  \cU_{P,Q}(M, N;\cS)} d_{nP+Q}.
\] 

Lemma~\ref{lem:height} shows that
\begin{equation}
\label{eq:W-LowBpund}
M^2\sharp\, \cU_{P,Q}(M, N;\cS)\ll \  \log W_d. 
\end{equation}  

For each prime $p$, denoting by $\alpha_p= \nu_p\(W_d\)$,
we have 
\begin{equation}
\label{eq:W-UpBpund}
\log W_d\leq \sum_{\substack{p\in \cS}} \alpha_p \log p. 
\end{equation} 

Note that by Lemma~\ref{lem:height},  every prime $p$ divides a term $d_{nP+Q}$ for $M<n\leq M+N$, with a power at most $\beta_p\ll (M+N)^2/\log p$. 
By Lemma~\ref{lem:divbym-den}, we then have
\begin{align*}
\alpha_p&\leq \sum_{k=1}^{\beta_p}\sharp\, \{M<n\leq M+N: p^k \mid  d_{nP+Q}\}\\
&\ll \sum_{k=1}^{\beta_p}\left(\frac{N}{\sqrt{k \log p}}+1\right)\ll \frac{N\sqrt{\beta_p}}{\sqrt{\log p}}+\beta_p\\
&\ll \frac{(M+N)^2}{\log p}. 
\end{align*}
Substituting this bound in~\eqref{eq:W-UpBpund} we obtain $\log W_d \ll (M+N)^2 S$. 
Now, recalling~\eqref{eq:W-LowBpund},  we complete the proof. 
\end{proof}

We have the straightforward consequence of Lemma~\ref{lem:sqx}, using dyadic partition as in~\cite[Corollary]{Shp}.

\begin{cor}
\label{cor:sqx}
For any finite set $\cS$ of primes of cardinality $S = \sharp\,  \cS$, we have 
\[
\sharp\,  \cU_{P,Q}(0, N;\cS) \ll  S \log  N.
\]
\end{cor}

We also  need the following version of a result of Silverman~\cite[Proposition~10]{Silv88}.

\begin{lem}\label{lem:PrimDiv}
Let $E$ be an elliptic curve given by a Weierstrass equation~\eqref{eq:Weier}  and $P\in E(\Q)$ a point which is not a torsion point. Then there exists a constant $c(P)$ depending only on  $P$ such that $d_{nP}$ has a primitive prime divisor for every $n>c(P)$.
\end{lem}

We note that since $P\in E(\Q)$ this also means that the constant $c(P)$ depends, 
implicitly, on the curve $E$. 

%
%
%
%
%

We also  consider the set of all points on $E(\Q)$ with $d_P \in \langle \cS\rangle$, that is, 
\[
\cU(\cS)=\{P \in E(\Q):~  d_{P} \in \langle \cS\rangle \}.
\]
By~\cite[Theorem~A]{S87} we know that $ \sharp\,\cU(\cS) \le \exp\(O(S)\)$. To get a linear (instead of exponential) dependence on $S$, we consider the set of points
\[
\cU(H;\cS)=\{P \in E(\Q):~\h(P)\le  H, d_{P} \in \langle \cS\rangle \}.
\]

Then, we have the following  version of Corollary~\ref{cor:sqx}. 

\begin{lem}\label{lem:siegel}
    For any finite set $\cS$ of primes of cardinality $S = \sharp\,  \cS$, we have 
\[
\sharp\,  \cU(H;\cS) \ll H^{(r-1)/2}S\log H.
\]
\end{lem}

\begin{proof}  The proof  follows same ideas as in the proof of Lemma~\ref{lem:divbypoverallpoints}, coupled with Corollary~\ref{cor:sqx}. Indeed, assume $E(\mathbb{Q})$ is of rank $r$ with a fixed basis 
 $\vec{P}=(P_1,\ldots, P_r)$ for its free part. As in the proof   of Lemma~\ref{lem:divbypoverallpoints}, for each $ i=1, \ldots, r$, let us denote 
\[
\widetilde \cU_{i}(H;\cS)=\bigcup_{|n_i| \le H^{1/2}} \bigcup_{\substack{P\in E(\Q)/\langle P_{i} \rangle\\ \| P\|_{\infty}\le  n_i}} \left\{Q=n_iP_{i}+P:~d_Q\in \langle \cS\rangle\right\},
\]
where  $E(\Q)/\langle P_{i} \rangle$ is defined by~\eqref{eq:E/Pi}.

Now, we see from~\eqref{eq:h and ||}  that there is a constant $c>0$, depending only on the 
choice of $\vec{P}$, such that 
\[
\sharp\,\cU(H;\cS) \le \sum_{i=1}^{r} \sharp\,\widetilde \cU_{i}(cH;\cS).
\]
Applying now Corollary~\ref{cor:sqx} to bound $\sharp\,\widetilde \cU_{i}(H;\cS)$ (we note that, by the definition of the set $\widetilde \cU_{i}(H;\cS)$ and the linear independence of the points $P_1,.\ldots,P_r$, for each $Q\in \cU_{i}(H;\cS)$, we have $\h(Q)\gg n_i^2$, and thus the bound~\eqref{eq:W-LowBpund} holds uniformly in $P$ and the implied constant in Corollary~\ref{cor:sqx} depends only on $\vec{P}$), we conclude the proof. 
\end{proof}

\subsection{Vertex covers}
We need the following graph-theoretic result, see~\cite[Lemma~2.7]{BHOS}.

\begin{lem}
\label{lem:GraphCover}
Assume we are given a  graph with the vertex set $\cV$ of cardinality  $\ell=\sharp\,  \cV$ and having no isolated vertex. Then there exists $ \cV_1\subseteq  \cV$ with $\sharp\,  \cV_1\leq \ell/2$ such that for any $v_2\in  \cV_2= \cV\setminus  \cV_1$ there exists a vertex $v_1\in  \cV_1$ which is a neighbour  of $v_2$.
\end{lem}

\section{Proof of Theorem~\ref{thm:MultDep-EC-MN-Q}}

\subsection{Classification of \md\ $s$-tuples}
\label{sec:Proof MNQ-1}
We only consider the case of $D^*_{\Psf,\Qsf}(M,N)$ as we have full analogues 
of all necessary ingredients to estimate $X^*_{\Psf,\Qsf}(M,N)$ in the identical way. 

Suppose that for some integers $n_1,\ldots,n_s \in [M+1,M+N]$  the terms $d_{n_1P_1 + Q_1}, \ldots, d_{n_sP_s+Q_s}$ are \md\ of maximal rank, that is, we have
\[
d_{n_1P_{1} + Q_1}^{k_1} \cdots d_{n_sP_{s}+Q_s}^{k_s}=1
\]
with some nonzero  integers $k_1,\ldots,k_s$.

Let $\rho_{i,p}$ be defined as in Section~\ref{sec:cong denom} and associated with $P_i$. 

Choose a positive real number $R \le N$ to be specified later, and let  $ \cW(R)$  be
the set of primes $p$ with $\rho_{i,p} \le R$ for at least one $i =1, \ldots, s$. 
By   Lemma~\ref{lem: rp a.a. p}  we have 
$\sharp\, \cW(R) \ll R^3/\log R$.

Write $t$ for the number of indices $i =1, \ldots, s$   for which $d_{n_iP_i + Q_i}$ 
has a prime divisor $p_i\notin \cW(R)$,
and let $r = s-t$ for the number of indices $i$ with $d_{n_iP_i + Q_i}$ having all prime divisors in $\cW(R)$.
Without loss of generality, we may assume that the corresponding integers are  $n_1,\ldots,n_t$, and $n_{t+1},\ldots,n_s$, respectively.

Applying Lemma~\ref{lem:sqx}, for $M\ge 1$, we obtain that the number $K_1$ of such $r$-tuples $\(n_{t+1},\ldots,n_s\) \in [M+1,M+N]^r$ satisfies
\begin{equation}
\label{eq:M1}
K_1\ll \(1+\frac{N}{M}\)^{2r}\(\frac{R^3}{\log R}\)^r.
\end{equation}
If $M=0$, by Corollary~\ref{cor:sqx}\ we have the bound
\begin{equation}
\label{eq:M1 M=0}
K_1\ll (\log N)^{r}\(\frac{R^3}{\log R}\)^r.
\end{equation}

We assume that such an $r$-tuple  $\(n_{t+1},\ldots,n_s\)$ is fixed.

Consider the $t$-tuples $\(n_1,\ldots,n_t\)\in [M+1,M+N]^t$. Recall that for any $1\leq i\leq t$, there is a prime $p_i\notin \cW(R)$ such that $p_i\mid d_{n_iP_i + Q_i}$.

Define the graph $\cG$ on $t$ vertices  $1, \ldots, t$ and connect the vertices $i$ and $j$ if and only if  $\gcd(d_{n_iP_i + Q_i},d_{n_jP_j + Q_j})$  has a prime divisor outside $\cW(R)$. 
 Observe that  as $d_{n_1P_1 + Q_1},\ldots,d_{n_sP_s + Q_s}$ are \md\ of
maximal rank, $\cG$ has no isolated vertex. Thus, by Lemma~\ref{lem:GraphCover}, there exists a
subset $\cI$ of $\{1,\ldots, t\}$ with
\[
m=\sharp\,  \cI\leq \fl{t/2}
\]
such that for any $j$ with
\[
j\in \{n_1,\ldots,n_t\}\setminus \cI
\]
the vertex $d_{n_jP_j + Q_j}$ is connected with some $d_{n_iP_i + Q_i}$ in $\cG$, for some $i\in \cI$.

Without loss of generality we may assume that $\cI=\{1,\ldots, m\}$. Trivially, the number $K_2$ of such $m$-tuples $(n_1,\ldots,n_m)\in[M+1,M+N]^m$ satisfies
\begin{equation}
\label{eq:M2}
K_2\ll N^m.
\end{equation}

We now fix  such an  $m$-tuple. For $\ell = t-m$, we now count  the number $K_3$ of
the remaining $\ell $-tuples $(n_{m+1},\ldots,n_t) \in [M+1,M+N]^\ell$.
Since each $d_{n_jP_j + Q_j}$ with $m+1\leq j\leq t$ has a common prime factor $p\notin \cW(R)$ with $d_{n_iP_i + Q_i}$ for some $1\leq i\leq m$, by Lemma~\ref{lem:divbyp} we obtain that $n_j$ comes from a set $\cN_j$ of cardinality
\[
\sharp\,  \cN_j \ll  N/\rho_{j,p} +1   \ll N/R +1 \ll N/R
\]
since we have assumes that $R \le N$. 
Thus we obtain
\begin{equation}
\label{eq:M3}
K_3 \le \prod_{j=m+1}^t \sharp\,  \cN_j \ll (N/R)^{t-m}.
\end{equation}

\subsection{Optimisation and deriving the bound on $D^*_{\Psf,\Qsf}(M,N)$}
\label{sec:optim}
If $M\le N$, then 
\[
D^*_{\Psf,\Qsf}(M,N)\le D^*_{\Psf,\Qsf}(0,2N).
\]
Putting this together with~\eqref{eq:M1 M=0},~\eqref{eq:M2} and~\eqref{eq:M3}, we obtain
\begin{align*}
D^*_{\Psf,\Qsf}(M,N)&\le K_1K_2K_3\\
& \ll (\log N)^{r} \(\frac{R^3}{\log R}\)^rN^mN^{t-m}R^{-(t-m)}\\
&\ll N^{t}R^{3s-7t/2}  \(\frac{\log N}{\log R}\)^{r},
\end{align*}
where the last inequality follows from the fact that $m\le t/2$. 
Writing $R=N^{\eta}$, with $0\leq \eta\leq 1$ to be chosen, we need to minimize the exponent (excluding $o(1)$) above, over the range $1\leq t\leq s$. The exponent is equal to
\[t+\eta(3s-7t/2)=t(1-7\eta/2)+3\eta s,
\]
which with  $\eta=2/7$ becomes $6s/7$. 
Hence, we have
\[
D^*_{\Psf,\Qsf}(M,N)\ll N^{6s/7}.
\]

If $M> N$, then  the bound~\eqref{eq:M1} becomes
\[
K_1\ll (R^3/\log R)^r, 
\]
and as above we obtain again
\begin{align*}
D^*_{\Psf,\Qsf}(M,N)&\le K_1K_2K_3\\
& \ll \(\frac{R^3}{\log R}\)^rN^mN^{t-m}R^{-(t-m)}\\
&\le N^{t}R^{3s-7t/2} (\log R)^{-r} .   
\end{align*} 
By the same choice of $R$ as above, we get
\[
D^*_{\Psf,\Qsf}(M,N)\ll N^{6s/7}, 
\]
and conclude the proof.

\subsection{Deriving the bound on $X^*_{\Psf,\Qsf}(M,N)$}
The argument is essentially identical to that used in the estimate on $D^*_{\Psf,\Qsf}(M,N)$, 
so we only sketch it.

Write $r$ for the number of indices $i =1, \ldots, s$ for which 
\[
\nu_p(x(n_iP_i + Q_i))\ne 0 \implies p\in \cW(R),
\]
where $\cW(R)$ is defined in Section~\ref{sec:Proof MNQ-1}. Without loss of generality, we may assume that the corresponding integers are $n_{t+1},\ldots,n_s$, respectively.

 Applying Lemma~\ref{lem:sqx}, for $M\ge 1$, we obtain that the number $K_1$ of such $r$-tuples $\(n_{t+1},\ldots,n_s\) \in [M+1,M+N]^r$ satisfies
\[
K_1\ll \(1+\frac{N}{M}\)^{2r}\(\frac{R^3}{\log R}\)^r.
\]
If $M=0$, by  Corollary~\ref{cor:sqx}, we have the bound
\[
K_1\ll (\log N)^{r}\(\frac{R^3}{\log R}\)^r.
\]
We assume that such an $r$-tuple  $\(n_{t+1},\ldots,n_s\)$ is fixed. 

 Consider the $t$-tuples $\(n_1,\ldots,n_t\)\in [M+1,M+N]^t$.  Since $x(n_1P_1 + Q_1),\ldots,x(n_sP_s + Q_s)$ are \md\ of maximal rank, for each $i=1,\ldots,t$, there is a prime $p_i\notin \cW(R)$ such that 
\[
\nu_{p_i}(x(n_iP_i + Q_i))\ne 0.
\]
and furthermore, there exists $1\le j\le t$, $j\ne i$, with 
\[
\nu_{p_i}(x(n_jP_j + Q_j))\ne 0.
\]

 We now define a new graph $\cG$ on $t$ vertices  $1, \ldots, t$ and connect the vertices $i$ and $j$ if and only if $\nu_p\(x\(n_iP_i + Q_i\)\),\nu_p\(x\(n_jP_j + Q_j\)\) \ne 0$ for some $p \notin \cW(R)$. By the observation above, $\cG$ has no isolated vertex. Thus, by Lemma~\ref{lem:GraphCover}, there exists a
subset $\cI$ of $\{1,\ldots, t\}$ with
\[
m=\sharp\,  \cI\leq \fl{t/2}
\]
such that for any $j$ with
\[
j\in \{n_1,\ldots,n_t\}\setminus \cI
\]
the vertex $j$ is connected with some $i$ in $\cG$, for some $i\in \cI$.

 Clearly by Lemma~\ref{lem:divbyp}, we have the similar estimates for the same quantities $K_2$ and $K_3$ as in the proof of \eqref{eqn:D^*}, and the proof concludes by the same argument as in Section~\ref{sec:optim}.
 
\section{Proof of Theorem~\ref{thm:MultDep-EC-MN-O}}

\subsection{Terms of eventually Zsigmondy sequences in finitely generated semigroups}

We say that a sequence of integers $\cZ = (z_n)_{n=1}^\infty$ is {\it eventually Zsigmondy\/} if 
there is some $N_0\ge 1$ such all terms $z_n$ with $n  \ge  N_0$ have a primitive prime divisor. 

We say  a finitely generated semigroup  $\Gamma\subseteq \Z$ is 
of rank $r$ if $r$ is the smallest number of  generators $g_1, \ldots, g_r$ such that 
\[
\Gamma = \{g_1^{k_1} \ldots g_r^{k_r}:~ k_i \in \Z, \ i =1, \ldots, r\}.
\]
Furthermore, we denote to $\ovG$ its division semigroup, that is, 
\[
\ovG = \{z \in \Z:~ z^m \in \Gamma\ \text{for some}\ m  \in \N \}.
\] 

\begin{lem}\label{lem:semigroup}
Let $\cZ = (z_n)_{n=1}^\infty$ be an eventually Zsigmondy sequence of integers 
and let $\Gamma\subseteq \Z$ be a finitely generated semigroup of rank $r\ge1$.
There is a constant $C(\cZ, r)$, depending only on $\cZ$ and $r$, such that 
\[
 \sharp\,\{n \in \N:~z_n \in \ovG\} \le C(\cZ, r). 
\]
\end{lem} 
 
\begin{proof} 
We show that one can take $ C(\cZ, r) = N_0 + r$, where $N_0$ is as in the definition of  an eventually Zsigmondy sequence. In other words, we show that the index $n$ of  $z_n\in \ovG$ can be chosen in at most $ N_0 + r$ ways.

Assume that 
\[
\sharp\,\{n \in N:~z_n \in \ovG\} > N_0 + r.
\]
Then  we can choose  $n_{i}$ for $i=0,1,\ldots, r$, 
with 
\[
N_0 \le n_{0}<n_{1}<\ldots < n_{r}  
\]
such that $z_{n_i} \in \ovG$.

We observe that since  $z_{n_i}$ has a  primitive divisor, we automatically 
conclude that $z_{n_i} \ne \pm 1$. In fact, we do not need $z_{n_0}$ to have
 a  primitive divisor, we only need  $z_{n_0}\ne \pm 1$, which we ensure 
 by the condition $n_0 \ge N_0$.

Let $g_1, \ldots, g_r$ be the generators of $\Gamma$. Then we have the 
following $r+1$ multiplicative relations
\begin{equation}\label{eq: zi g1...gr}
z_{n_i}^{m_i} =g_1^{k_{1,i}} \ldots  g_{r}^{k_{r,i}} , \qquad i=0,\ldots,r, 
\end{equation}
with some nonzero vectors  $\uk_i=(k_{1,i},\ldots, k_{r,i})\in \Z^{r}$ 
and a positive integer $m_i$.

Clearly, we can find a non-zero integer vector $\ut=(t_0,\ldots,t_r)$ such that 
\begin{equation}\label{eq:znrelation}
z_{n_0}^{m_0t_0}\ldots z_{n_r}^{m_rt_r}=1.
\end{equation}
Indeed,  $\ut$ is any non-zero solution to  a system of $r$   linear homogeneous  equations with all integer coefficients (given by the exponents in~\eqref{eq: zi g1...gr}), and in $r+1$ variables.

However, a relation of the form \eqref{eq:znrelation} cannot hold. Indeed, if at least two coordinates of $\ut$ are non-zero, then the Zsigmondy property is clearly violated. On the other hand, if exactly one $t_i$ is non-zero, then \eqref{eq:znrelation} cannot hold, since each $z_{n_i}\ne \pm 1$.
\end{proof}

We emphasise that it is very important that the constant $C(\cZ, r)$ in 
Lemma~\ref{lem:semigroup} depends only on the rank of the semigroup $\Gamma$ 
rather than on its generators.

\subsection{Concluding the proof} Because of the inequality~\eqref{eq:D D*}, 
it is enough to estimate  
\begin{align*}
\Dsf^*_{\Psf}(M,N) = &\sharp\, \left\{ (n_1,\ldots, n_s)\in \left(M,M+N\right]^s\ : \right. \\ 
&\quad\quad\quad\left. d_{n_1P_1},\ldots,d_{n_sP_s} \  \text{are \md\ of maximal rank} \right\}. 
\end{align*}

Hence, we  estimate the number of $s$-tuples $(n_1,\ldots, n_s)$ in the box $\left(M,M+N\right]^s$ such that
\begin{equation}
\label{eq:mult rel}
d_{n_1P_1}^{k_1} \ldots d_{n_sP_s}^{k_s}=1
\end{equation}
for some $k_1, \ldots, k_s\in \Z \setminus \{ 0 \}$. Fix the first $s-1$ coefficients $n_1,\ldots ,n_{s-1}$, 
and rewrite~\eqref{eq:mult rel} as
\[
d_{n_sP_s}^{-k_s} = d_{n_1P_1}^{k_1} \ldots d_{n_{s-1}P_{s-1}}^{k_{s-1}}
\]
with $k_s \ne 0$. 

Hence,  we see that $d_{n_sP_s}$ belongs to the  division semigroup generated by $d_{n_1P_1},  \ldots, d_{n_{s-1}P_{s-1}}$. 

Since by Lemma~\ref{lem:PrimDiv},  the sequence $d_{n P_s}$ is eventually Zsigmondy, the  bound 
$\Dsf^*_{\Psf}(M,N) \ll N^{s-1}$ now follows 
from Lemma~\ref{lem:semigroup}, which concludes the proof.

\section{Proof of Theorem~\ref{thm:mdoveralldenos}} 
\subsection{Preliminary comments} 
We argue exactly as in the proof of Theorem~\ref{thm:MultDep-EC-MN-Q}. First, we fix a basis vector 
$\vec{P}_i$ for the free part of $E_i(\Q)$  for each $i=1,\ldots, s$. Now, we choose a positive real number $R \le H^{1/2}$ to be specified later, and let  $\vec{W}(R)$  be the set of primes $p$ with $\rho_p\({\vec{P}_i}\) \le R$ for at least one $i =1, \ldots, s$, where $\rho_p\({\vec{P}_i}\) $ is defined in~\eqref{eqn:rE} with respect to the curve $E_i$. As before, by Lemma~\ref{lem: rp a.a. p}, we have 
$\sharp\, \vec{W}(R) \ll R^3/\log R$.

\subsection{Estimation of $\whD^*(H)$}
\label{eq:estim DH}
Write $t$ for the number of indices $i =1, \ldots, s$   for which $d_{P_i}$ has a prime divisor $p_i\notin \vec{W}(R)$,
and thus we have $s-t$ indices $i$ with $d_{P_i}$ having all prime divisors in $\vec{W}(R)$.
Without loss of generality, we may assume that the corresponding indices are $1,\ldots,t$, and $t+1,\ldots,s$, respectively.

Applying Lemma~\ref{lem:siegel}, we get that the number $K_1$ of such $(s-t)$-tuples $\(P_{t+1},\ldots,P_s\)$ of points is
\begin{equation}\label{eq:K1}
K_1\le H^{(r_{t+1}+\cdots+r_s-(s-t))/2}(\log H)^{s-t}\(\frac{R^3}{\log R}\)^{s-t}.
\end{equation}
Assume that such an $(s-t)$--tuple of points $\(P_{t+1},\ldots,P_s\)$ is fixed.

For each remaining $t$-tuples $\(P_1,\ldots,P_t\) \in E_1(\Q) \times \ldots \times E_t(\Q)$ of the corresponding canonical height at most $H$,  and for any $1\leq i\leq t$, there is a prime $p_i\notin  \vec{W}(R)$ such that $p_i\mid d_{P_i}$. Again, consider the graph $\cG$ on $t$ vertices  $1, \ldots, t$, and connect the vertices $i$ and $j$ if and only if  $\gcd(d_{P_i},d_{P_j})$ has a prime divisor outside $ \vec{W}(R)$. 

Since $d_{P_1},\ldots,d_{P_s}$ are \md\ of
maximal rank, $\cG$ has no isolated vertex. Thus, by Lemma~\ref{lem:GraphCover}, there exists a
subset $\cI$ of $\{1,\ldots, t\}$ with
\[
m=\sharp\,  \cI\leq \fl{t/2}
\]
such that for any $j$ with
\[
j\in \{1,\ldots,t\}\setminus \cI,
\]
the vertex $j$ is connected with some $i$ in $\cG$, for some $i\in \cI$. Without loss of generality, let us write that $\cI=\{1,\ldots, m\}$. Trivially, the number $K_2$ of such $m$-tuples of points is
\[
K_2\ll H^{(r_1+\ldots + r_m)/2},
\]
see~\eqref{eq:Count height}. 

We now fix  such an  $m$-tuple of points. For $\ell = t-m$, we now count  the number $K_3$ of
the remaining $\ell $-tuples of points $(P_{m+1},\ldots,P_t)$. Since each $d_{P_j}$ with $m+1\leq j\leq t$ has a common prime factor $p\notin  \vec{W}(R)$ with $d_{P_i}$ for some $1\leq i\leq m$, by Lemma~\ref{lem:divbypoverallpoints}, we obtain that $P_j$ comes from a set $\cP_j$ of cardinality
\[
\sharp\,  \cP_j \ll H^{(r_j-1)/2}\(H^{1/2} /R + 1\) \ll H^{r_j/2} /R  
\]
since we have assumed that $R \le H^{1/2}$.

Thus we obtain
\[
K_3 \le \prod_{j=m+1}^t \sharp\,  \cP_j \ll  H^{(r_{m+1}+ \ldots + r_t)/2}/ R^{t-m}.
\] 

Therefore, recalling that $m\leq \fl{t/2}$, we have
\begin{equation}\label{eq:K2K3}
K_2K_3\ll H^{(r_1+ \ldots + r_t)/2}/R^{\rf{t/2}}.
\end{equation} 

Putting~\eqref{eq:K1} and~\eqref{eq:K2K3} together, we obtain
\[
\whD^*(H)\ll H^{(r_{1}+ \ldots + r_t)/2-(s-t)/2}R^{3s-7t/2}(\log H/\log R)^{s-t}.
\]
As in Section~\ref{sec:optim}, writing $R=H^{\eta}$, with $0\leq \eta\leq 1$ to be chosen, we need to minimize the exponent (excluding $o(1)$) above, over the range $1\leq t\leq s$. The exponent is
\[
(r_{1}+ \ldots + r_t)/2+s(3\eta-1/2)+t(1/2-7\eta/2)
\]
 and choosing $\eta=1/7$, we obtain the desired bound on $\whD^*(H)$.

\subsection{Estimation of $\whX^{*}(H)$}
As in the proof of Theorem~\ref{thm:MultDep-EC-MN-Q}, but also appealing to the arguments in Section~\ref{eq:estim DH}, in particular to full analogies of the bounds~\eqref{eq:K1} and~\eqref{eq:K2K3}.  We skip the details as they are 
easy to recover.

\section{Further questions}

First we observe that it is highly likely that  one can extend Theorems~\ref{thm:MultDep-EC-MN-Q} 
and~\ref{thm:MultDep-EC-MN-O} to number and function fields. 

Examining the proof of Theorem~\ref{thm:MultDep-EC-MN-O}, one can easily see that
it can be extended to $D_{\Psf,\Qsf}(M,N)$ where all components of $\Qsf$ are torsion points 
on corresponding elliptic curves. This is thanks to the generalisation of Lemma~\ref{lem:PrimDiv}
given by Verzobio~\cite{Ver20,Ver21,Ver23}. 
In fact for the bound on  $D^*_{\Psf,\Qsf}(M,N)$ we need only one component of $\Qsf$ to be a
torsion point.
We also note that  Lemma~\ref{lem:semigroup} can be extended into a much broader context 
of commutative rings. 


Finally, partially motivated by the results of~\cite{HLS} and partially by our results, 
we ask  about an upper bound on  the number of $s$-tuples $(n_1,\ldots,n_s)$ with entries from an interval $(M,M+N]$
and 
such that the product $ d_{n_1P_1+Q_1} \cdots d_{n_sP_s+Q_s} $  is a perfect power, and similarly 
for $x(n_1P_1+Q_1) \cdots x(n_sP_s+Q_s)$.

Finally, inspired by~\cite[Corollary~1.2]{BarSha} and our theme of the results, we also ask  the following:

\begin{question}
Let $E$ be an elliptic curve over $\Q$ and let $P \in E(\Q)$ be a fixed non-torsion point. Assume that $\fsf = \(f_1, \ldots, f_s\) \in \Q(X,Y)^s$ are  $s$ multiplicatively independent, non-zero rational functions Then, can we estimate the following:
\begin{align*}
F^*_{\fsf,\Psf,\Qsf}(M,N)=\sharp\, \{ n&\in   (M,M+N]:\\
&\quad f_1(nP), \ldots, f_s(nP)\ \text{are \md\ of maximal rank} \}?
\end{align*}
\end{question}

\section*{Acknowledgement}

During this work,  A.B. and L.H. were  supported, in part, by the
NKFIH grants 130909 and 150284 and  S.B,, A.O.  and I.S.  by the Australian Research Council Grant  DP230100530. 

\appendix
\section{Multiplicative dependence of numerators}
\label{app:numer}

Here we merely develop some tools to enable us to extend our results to numerators 
of rational points on elliptic curves. We do this under a mild additional condition 
on the coefficients of the Weierstrass equation~\eqref{eq:Weier}.

We also have a variant of Lemma~\ref{lem:divbym-den} for numerators. 

\begin{lem}\label{lem:divbym-num}  Assume that $E$ is given by~\eqref{eq:Weier} with $b \ne 0$.
For a prime  $p$ and an integer  $k\ge 1$,  uniformly over $M\geq 0$, we have
\[
\sharp\,\{M<n\leq M+N :~p^k\mid a_{nP+Q}\}\ll \frac{N}{\sqrt{k \log p}}+ 1.
\]
\end{lem}

\begin{proof} Note that if $a_{nP+Q}\equiv 0 \mod {p^k}$, then from the Weierstrass equation~\eqref{eq:Weier}  we conclude that 
\[
y(nP+Q)^2   \equiv  b \mod {p^k}.
\]
 Note that since $p\mid a_{nP+Q}$ 
 we have $p\nmid d_{nP+Q}$ and thus $y(nP+Q)$ is well defined modulo $p^k$.
Since $b$ is fixed it is easy to show that there are at most $C$ possible values of $y(nP+Q)$ modulo $p^k$ where $C$ depends only on $b \ne 0$ (one can also simply use a much more general result of Huxley~\cite{Hux}). Let $T$ be the cardinality in the statement that we want to bound. If $T \le  C + 1$ there is nothing to prove. Otherwise we partition the interval $(M, M+N]$ into  $L = \rf{T/(C+1)} - 1$ semi-open  intervals of the shape $(u, u+h]$ 
 of equal length $h = N/L$. Clearly one such interval has to contain at least $C+1$ values of $n$ with $p^k\mid a_{nP+Q}$.   
 
However, since there are at most $C$ values of $y(nP+Q)$ modulo $p^k$, then there is an interval $(u, u+h] \subseteq (M, M+N]$ containing two  
integers $n_1<n_2$ with
\begin{equation}
\label{eq:x,y mod pk}
\begin{split}
& x(n_1P+Q)\equiv x(n_2P+Q) \equiv 0  \mod {p^k},   \\
& y(n_1P+Q)\equiv y(n_2P+Q)  \mod {p^k}. 
\end{split}
\end{equation}
Then we have 
\begin{equation}
\label{eq: nu_p(d)}
\nu_p\(d_{(n_2-n_1)P + Q}\)\gg k.
\end{equation}
To prove this claim, note that by the standard formula of addition of points, we have
\[x((n_1-n_2)P)=\frac{(y_2+y_1)^2-(x_1+x_2)(x_1-x_2)^2}{(x_1-x_2)^2},
\]
where we write $n_1P+Q=(x_1,y_1)$ and $n_2P+Q=(x_2,y_2)$.   
Now, we may consider $p^k$ sufficiently large such that $\nu_p(4b)<k$, since otherwise the result follows.
Using the Weierstrass equation~\eqref{eq:Weier}, the first congruence of~\eqref{eq:x,y mod pk} and standard properties of valuations, we obtain  that $\nu_p(2y_i)<k/2$, $i=1,2$. This, together with the second congruence of~\eqref{eq:x,y mod pk} (writing $y_1+y_2 - (y_1-y_2) = 2y_2$), implies that 
 $\nu_p(y_1+y_2)<k/2$, which proves~\eqref{eq: nu_p(d)} as $b \neq 0$ is fixed, and $\nu_p(x_1-x_2)\geq k$.

Using the bound~\eqref{leq:height-num}  we see from~\eqref{eq: nu_p(d)} that 
\[
k  \log p \ll (n_2-n_1)^2 \le h^2 \le (N/L)^2, 
\]
and the result follows
\end{proof}

Next we consider the set
\[
\cV_{P,Q}(M, N;\cS)=\{M<n\leq M+N:~  a_{nP+Q}\in \langle \cS\rangle\}.
\]

We note that 
\[
\frac{\log |a_{nP + Q}|}{\log d_{nP + Q}^2} \to 1
\]
as $n \to \infty$, see~\cite[Section~IX.3]{Silv-Book}, hence by Lemma~\ref{lem:height}, 
or any    fixed points $P,Q \in E(\Q)$, we have  
\begin{equation}
\label{leq:height-num} 
\log \(|a_{nP + Q}|\)=  \(\h(P) + o(1)\) n^2,  
\end{equation}
as $n \to \infty$. 

Using the bound~\eqref{leq:height-num} and Lemma~\ref{lem:divbym-num} 
(instead of   Lemma~\ref{lem:height} and Lemma~\ref{lem:divbym-den}, respectively), 
we immediately establish an analogue  of Lemma~\ref{lem:sqx}. 

\begin{lem}\label{lem:sqx-num} Let $E$ be given by~\eqref{eq:Weier} with $b \ne 0$. 
For any finite  set $\cS$ of primes of cardinality $S = \sharp\,  \cS$, we have  
\[
\sharp\,  \cV_{P,Q}(M,N,\cS)\ll  \left(1+\frac{N}{M}\right)^2 S.
\]
 \end{lem} 

As before, we see hat  Lemma~\ref{lem:sqx-num}, applies

\begin{cor}
\label{cor:sqx-num}
Let $E$ be given by~\eqref{eq:Weier} with $b \ne 0$. 
For any finite  set $\cS$ of primes of cardinality $S = \sharp\,  \cS$, we have  
\[
\sharp\,  \cV_{P,Q}(0, N;\cS) \ll  S \log  N.
\] 
\end{cor}

\end{document}